\documentclass[12pt]{article}
\usepackage{amsmath,amsthm,amssymb,amsfonts,mathrsfs}
\usepackage[hypertex, bookmarks, colorlinks=true, plainpages = false, citecolor = green, urlcolor = blue, filecolor = blue]{hyperref}
\newtheorem{theorem}{Theorem}[section]
\newtheorem{lemma}[theorem]{Lemma}
\newtheorem{proposition}[theorem]{Proposition}
\newtheorem{corollary}[theorem]{Corollary}
\newtheorem{definition}[theorem]{Definition}
\newtheorem{example}[theorem]{Example}

\def\a{\bar{a}}\def\b{\bar{b}}
\def\x{\bar{x}}\def\y{\bar{y}}

\def\Zn{\mathbb Z}
\def\Qn{\mathbb Q}
\def\Rn{\mathbb R}

\def\0{\sf 0}

\makeatletter
\def\dotminussym#1#2{%
  \setbox0=\hbox{$\m@th#1-$}%
  \kern.5\wd0%
  \hbox to 0pt{\hss\hbox{$\m@th#1-$}\hss}%
  \raise.8\ht0\hbox to 0pt{\hss$\m@th#1.$\hss}%
  \kern.5\wd0}

\makeatletter
\newcommand*{\defeq}{\mathrel{\rlap{%
                     \raisebox{0.3ex}{$\m@th\cdot$}}%
                     \raisebox{-0.3ex}{$\m@th\cdot$}}%
                     =}

\begin{document}
\begin{center}
{\Large\sc{Affinization and quantifier-elimination}}\vspace{6mm}  

{\bf Seyed-Mohammad Bagheri}
\vspace{4mm}

{\footnotesize {\footnotesize Department of Pure Mathematics, 
Tarbiat Modares University,\\ Tehran, Iran, P.O. Box 14115-134}}
\vspace{1mm}

bagheri@modares.ac.ir 
\end{center}

\begin{abstract}
Quantifier-elimination or model-completeness of the affine part of some classical
first order theories are proved.
\end{abstract}


\bigskip

A powerful technic in applied model theory is quantifier-elimination.
It provides an easy understanding of first order theories as well as simplification of proofs.
This technic is used in continuous logic too. The theory of Hilbert spaces is a typical example in this respect.
Continuous logic \cite{BBHU, CK} is an extension of first order logic and equips it with the notion of metric.
Affine logic is a sublogic of continuous logic having a parallel (but distinct)
type of model theory (\cite{Bagheri-Lip, ACMT, Ibarlucia}).
There are many interesting affine theories having quantifier-elimination.
The goal of the present paper is to find some new examples. A good method for this purpose is affinization.
Some continuous theories transfer model theoretic properties to their affine counterpart.
An easy example in this respect is the theory of atomless probability algebras.
Its affine part is the theory of probability algebras. They are both stable and have quantifier-elimination.
In this paper, we concentrate on the affinization of classical first order theories
such as ODAG, ACF$_p$ and RCF (see \cite{Marker}).

\section{Introduction}
Affine logic (AL) is a well-behaved fragment of continuous logic (CL).
Let $L$ be a first order language. To every relation symbol $R$ (resp. function symbol $F$)
is assigned a Lipschitz constant $\lambda_R\geqslant0$ (resp. $\lambda_F\geqslant0$).
Terms are defined as in first order logic. Affine formulas are inductively defined by
$$1,\ \ \ R(t_1,...,t_n),\ \ \ \phi+\psi,\ \ \ r\phi,\ \ \ \sup_x\phi,\ \ \ \inf_x\phi$$
where $R$ is a $n$-ary relation symbol, $t_1,...,t_n$ are terms and $r\in\Rn$.
The metric symbol $d$ is considered as a binary relation symbol with $\lambda_d=1$. So, $d(t_1,t_2)$ is formula.
In CL, one further allows $\phi\wedge\psi$ and $\phi\vee\psi$ as formulas.
So, every notion used in this paper has both CL and AL variants.
In this text, unless otherwise stated, by formula we mean affine formula.
A structure $M$ in $L$ is a metric space such that for $n$-ary $R$ and $n$-ary $F$, the maps
$$R^M:M^n\rightarrow[0,1],\ \ \ \ \ \ \ \ \ \ F^M:M^n\rightarrow M$$
are $\lambda_R$-Lipschitz and $\lambda_F$-Lipschitz respectively, where
$$d(\x,\y)=\sum_{i=1}^nd(x_i,y_i).$$

A condition is an expression of the form $\phi\leqslant\psi$.
A theory is a set $T$ of closed (without free variables) conditions.
Notions such as $M\vDash\phi\leqslant\psi$,\ \ $M\vDash T$,\ \ $T\vDash\phi\leqslant\psi$
and satisfiability of $T$ are all defined in the usual way (for both CL and AL theories).

\begin{theorem}
(Affine compactness) Assume for every $0\leqslant\phi_1$, ..., $0\leqslant\phi_n$ in $T$ and
$0\leqslant r_1,...,r_n$, the condition $0\leqslant\sum_ir_i\phi_i$ is satisfiable. Then, $T$ is satisfiable.
\end{theorem}

An immediate application of the affine compactness theorem is that any theory having a model of cardinality
at least two has models of arbitrarily large cardinalities. An other easy consequence is the following lemma.
The affine closure of $T$ is the collection of conditions $0\leqslant r_1\sigma_1+\cdots+r_k\sigma_k$
where $0\leqslant r_i$ and $0\leqslant\sigma_i$ belongs to $T$.

\begin{lemma} \label{proofs}
Let $T$ be a set of affine conditions of the form $0\leqslant\sigma$.

\emph{(i)} If $T\vDash0\leqslant\theta$, then for each $\epsilon>0$ there exists
$0\leqslant\sigma$ in the affine closure of $T$ such that \ $\emptyset\vDash\sigma\leqslant\theta+\epsilon$.
In particular, $0\leqslant\sigma\vDash0\leqslant\theta+\epsilon$.

\emph{(ii)} If $0\leqslant\sigma\vDash0\leqslant\theta$, for each $\epsilon>0$ there exists
$\delta>0$ such that $-\delta\leqslant\sigma\vDash-\epsilon\leqslant\theta$.
\end{lemma}

The affine counterpart of ultraproduct method is the ultramean method (see \cite{ACMT}).
Let $M_i$ be a $L$-structure for each $i\in I$ and $\mu$ be a finitely additive
probability measure on the power set of $I$.
Define a pseudo-metric on $\prod_{i}M_i$ by setting $$d(a,b)=\int_{i\in I} d(a_i,b_i)d\mu.$$
Denote the resulting metric space by $(M,d)$ and the class of $(a_i)$ by $[a_i]$.
Then, $M$ is a $L$-structure by setting (e.g. for unary relations and operations):
$$F^M([a_i])=[F^{M_i}(a_i)]$$$$R^M([a_i])=\int R^{M_i}(a_i)d\mu.$$
One proves that for every affine sentence $\phi$, $$\phi^M=\int\phi^{M_i}d\mu.$$
The resulting structure is denoted by $\prod_{\mu}M_i$.
So, if $M_i$ is a model of an affine theory for each $i$, then so is $\prod_{\mu}M_i$.
In the case $I=\{0,1\}$ and $\mu(0)=\frac{1}{2}$, this structure is denoted by $\frac{1}{2}M_1+\frac{1}{2}M_2$.

Let $T$ be an affine theory in $L$ and $|\x|=n$. $D_n(T)$ denotes the partially ordered vector space of formulas
$\phi(\x)$ up to $T$-equivalence, where we define $0\leqslant\phi$ if $T\vDash0\leqslant\phi$.
A $n$-type is a positive linear map $p:D_n(T)\rightarrow\Rn$ such that $p(1)=1$.
If $\a\in M\vDash T$, then $$\phi(\x)\mapsto\phi^M(\a)$$ defines a type.
A type is realized in $M$ if it is of this form for some $\a\in M$. The set of all $n$-types of $T$
is denoted by $K_n(T)$. It is endowed with the weak-star topology induced by $D_n(T)^*$.
Type spaces are compact and convex. It is natural to see that convexity creates
new notions in model theory.
A type $p$ is \emph{extreme} if for every types $p_1,p_2$ and $0<\lambda<1$,
one has that $p=p_1=p_2$ whenever $p=\lambda p_1+(1-\lambda)p_2$.
The set of extreme $n$-types of $T$ is denoted by $E_n(T)$.
This is a non-empty set by Krein-Milman theorem. A model $M\vDash T$ is \emph{extremal} if it realizes only
extreme types. Such models exist always \cite{Bagheri-Extreme}.
If $T$ is affinely complete, its extremal models form an elementary class in the
CL sense if and only if $E_n(T)$ is closed in $K_n(T)$
for every $n$. In this case, the extremal theory is denoted by $T^{\textrm{ex}}$.
Conversely, the \emph{affine part} of a CL-theory $\mathbb{T}$, denoted by $\mathbb{T}_{\textrm {af}}$,
is the set of all affine conditions satisfied by $\mathbb{T}$, i.e.
$$\mathbb{T}_{\textrm {af}}=\{\phi\leqslant\psi:\ \mathbb{T}\vDash\phi\leqslant\psi,\ \ \ \ \phi,\psi\ \textrm{are affine}\}.$$
If $\mathbb{T}$ is complete in the CL sense, $\mathbb{T}_{\textrm {af}}$
is complete in the AL sense, i.e. for every affine sentence $\phi$ there is a unique $r$
such that $\mathbb{T}_{\textrm {af}}\vDash\phi=r$.

An affine theory $T$ if \emph{Bauer} if $K_n(T)$ is a Bauer simplex for every $n\geqslant1$.
For every complete first order theory $\mathbb{T}$,\ \ $\mathbb{T}_{\textrm{af}}$ is Bauer
(\cite{Ibarlucia}, Th. 26.9).
So, $(\mathbb{T}_{\textrm{af}})^{\textrm{ex}}=\mathbb{T}$ and hence
$E_n(\mathbb{T}_{\textrm{af}})=S_n(\mathbb{T})$ (the classical type space for $\mathbb{T}$).

A function $f:K_n(T)\rightarrow\Rn$ is affine if for every $p,q$ and $0\leqslant\lambda\leqslant1$
$$f(\lambda p+(1-\lambda)q)=\lambda f(p)+(1-\lambda)f(q).$$
The set of affine continuous functions on $K_n(T)$ is denoted by $\mathbf{A}(K_n(T))$.
For every formula $\phi(\x)$ where $|\x|=n$, set $$\hat\phi:K_n(T)\rightarrow\Rn$$
$$\hat\phi(p)=p(\phi).$$
Such functions are affine continuous and form a dense subset of $\mathbf{A}(K_n(T))$.
In fact, any subspace of $\mathbf{A}(K_n(T))$ which contains constant maps and separates points
is dense.

\begin{definition} \em{A theory $T$ has \emph{quantifier-elimination} if for every formula $\phi(\x)$ and $\epsilon>0$
there is a quantifier-free formula $\psi(\x)$ such that $T\vDash|\phi-\psi|\leqslant\epsilon$.}
\end{definition}

An \emph{infimal formula} is a formula of the form $\inf_{\y}\phi$ where $\phi$ is quantifier-free.

\begin{proposition}\label{separation}
A theory $T$ has quantifier-elimination if and only if every type is determined by quantifier-free formulas,
i.e. if $p(\theta)=q(\theta)$ for every quantifier-free (or equivalently atomic) $\theta$, then $p=q$.
Similarly, $T$ is model-complete if and only if types are determined by infimal formulas.
\end{proposition}
\begin{proof}
Assume for every distinct types $p,q\in K_n(T)$ there is a quantifier-free $\theta$ such that
$p(\theta)\neq q(\theta)$.
Let $B$ be the vector space of functions $\hat\theta$ where $\theta(x)$ is quantifier-free.
Then $B$ is a subspace of $\mathbf{A}(K_n(T))$ which contains the constant functions and separates points.
So, it is dense in $\mathbf{A}(K_n(T))$. In particular, every $\hat\phi$ is approximated by
functions of the form $\hat\theta$ where $\theta$ is quantifier-free. In other words,
$\phi$ is approximated by such formulas. So, $T$ has quantifier-elimination.
The other direction is obvious.

For the second part note that $T$ is model-complete if and only if every formula
is approximated by infimal formulas. Then, a similar argument works in this case.
\end{proof}

So, $T$ has quantifier-elimination if and only if quantifier-free formulas separate types,
i.e. if $p\neq q$, then there is a quantifier-free formula $\theta$ such that $\hat\theta(p)\neq\hat\theta(q)$.
Also, $T$ is model-complete if and only if infimal formulas separate types.
Similar results hold in continuous logic where one uses Stone-Weierstrass theorem in the proof.
In particular, a continuous theory is model-complete if and only if infimal (in the CL sense) formulas separate types.

For certain classical theories quantifier-elimination and model-completeness are transferred to the affine part.
Although these affine parts are usually decidable, it is not easy to give a concrete axiomatization for them.
We start with a direct proof of quantifier-elimination for the affine part of the first order theory of any
one dimensional finite vector space. As stated above, this theory has infinite models.
\bigskip

\noindent{\bf Vector spaces over a finite field}\\
Let $\mathbb F$ be a finite field with $|\mathbb F|=q$ and $L=\{+,\ 0,\ \alpha\cdot\}_{\alpha\in\mathbb F}$.
Let $\mathbb{T}$ be the first order theory $\mathbb F$ as a vector space over $\mathbb F$.
So, $\mathbb T$ is a complete theory and $\mathbb F$ is its unique model.
Hence, $T=\mathbb{T}_{\textrm{af}}$ is complete in the AL sense.
Models of $T$ are generally vector spaces over $\mathbb F$ with nontrivial metrics.
We prove that $T$ has quantifier-elimination. Note that $L$-terms are of the form $a_1x_1+\cdots+a_kx_k$.
For $x\in M\vDash T$ let $|x|=d(x,0)$. Also, for $\a\in\mathbb F^n$ and tuple $\x$ of variable
let $\a\cdot\x=\sum_{i=1}^n a_ix_i$.
Then, every quantifier-free affine formula $\phi(\x)$ is a linear combination
of formulas of the form $|\a\cdot\x|$, i.e.
$$\ \ \ \ \ \phi(\x)=r+\sum r_\ell\ |\a_\ell\cdot\x|,
\ \ \ \ \ \ \ \ \a_\ell\in\mathbb F^n, \ \ r, r_\ell\in\Rn.$$

\begin{proposition} \label{qe}
$T$ has elimination of quantifiers.
\end{proposition}
\begin{proof}
We show that every formula is equivalent to a quantifier-free formula.
It is sufficient to show that this holds in $\mathbb F$.
Consider $\mathbb F^n$ as a vector space over $\mathbb F$ and let $\b_1,...,\b_m\in\mathbb F^n$
be a maximal list of pairwise linearly independent elements of $\mathbb F^n$.
Then, $m=\frac{q^n-1}{q-1}$. This is the number of $1$-dimensional subspaces of $\mathbb F^n$ which
is also equal to the number of $(n-1)$-dimensional subspaces of $\mathbb F^n$. We assume
$$\sup_y\sum r_\ell\ |\a_\ell\cdot\x-y|=s_0+\sum_{\ell=1}^m s_\ell\ (1-|\b_\ell\cdot\x|)$$
and find the coefficients $s_\ell$ so that the equality holds for every $\x\in\mathbb F^n$.
Clearly, it is sufficient to verilfy that the equality hold for the values $\x=0,\b_1,...,\b_m$.
Putting these tuples in the condition, we obtain $m+1$ linear equations with
indeterminates $s_0,s_1,...,s_{m}$. We must show that this system of equations has a solution.
For this purpose, we have only to show that the $(m+1)\times(m+1)$ matrix with entries:
\[ A_{\ell k}= \left\{
  \begin{array}{ll}
    1 \ \ \ & \hbox{if}\ k=1 \\
    1-|\b_\ell\cdot\b_k| \ \ \ \ \ \ & \hbox{if}\ k\neq 1
  \end{array}
\right.\]
is invertible. In fact, we only need to show that the $m\times m$ matrix
$U=[u_{\ell k}]$ where $u_{\ell k}=1-|\b_\ell\cdot\b_k|$ is invertible.
This is the matrix of incidence between $1$-dimensional and $(n-1)$-dimensional
subspaces of $\mathbb F^n$. By a result of W. M. Kantor \cite{W.Kantor}, $U$ has rank $m$.
\end{proof}

Replacing finite fields with the rational field $\mathbb Q$, one obtains the theory
of nontrivial torsion-free divisible Abelian groups DAG.
\bigskip

\noindent{\bf Question:} Prove that the affine part of DAG has quantifier-elimination.

\section{Rich affinization}
In this section, we prove that the affine part of a first order theory with quantifier-elimination
has quantifier-elimination provided that the set of atomic formulas is sufficiently rich.
More precisely, we assume that every finite disjunction (or conjunction) of atomic formulas is equivalent
to an atomic formula. This property does not hold DAG or the example given in the previous section.
Let $\mathbb{T}$ be a complete first order theory and $\mu$ be a regular Borel probability measure on
$S_n(\mathbb{T})$.
As in the affine (or continuous) case, we may regard first order types as functions on formulas,
i.e. $u(\phi)=1$ if $\phi\in u$ and $u(\phi)=0$ otherwise, where $u\in S_n(\mathbb{T})$.
So, for a first order $\phi$ let $\hat{\phi}(u)=u(\phi)$ and define
$$\mu(\phi)\defeq\mu\{u\in S_n(\mathbb{T}): \phi\in u\}=\int\hat\phi(u)d\mu.$$
By regularity, $\mu$ is uniquely determined by its values on such sets.
As stated above, $T=\mathbb{T}_{\textrm{af}}$ is an AL-complete theory.
Moreover, $\mathbb{T}=T^{\textrm{ex}}$ and $S_n(\mathbb{T})=E_n(T)$.
In this case, and we say that two probability measures $\mu,\nu$ on $S_n(\mathbb{T})$ coincide on a
first order formula $\phi(\x)$ if $\mu(\phi)=\nu(\phi)$. 
By the inclusion-exclusion principle, for every first order formulas $\phi_1,...,\phi_n$ one has that
$$\mu(\bigvee_{i=1}^n\phi_i)=\sum_{\emptyset\neq J\subseteq\{1,...,n\}}(-1)^{|J|+1}\mu(\bigwedge_{j\in J}\phi_j).$$
In fact, this is a consequence (i.e. by integrating) of the more general equality holding for any $f_1,...,f_n$ in a Riesz space:
$$\bigvee_{i=1}^nf_i=\sum_{\emptyset\neq J\subseteq\{1,...,n\}}(-1)^{|J|+1}\bigwedge_{j\in J}f_j.$$
The duals of these equalities hold similarly.

\begin{lemma}\label{atomic values}
Let $\mathbb T$ be a complete first order theory which has quantifier-elimination and $\mu$, $\nu$
be regular Borel probability measures on $S_n(\mathbb{T})$ which coincide on atomic formulas.
Assume $\eta\vee\theta$ is $\mathbb{T}$-equivalent to an atomic formula whenever $\eta,\theta$ are atomic.
Then $\mu=\nu$. Similar result holds if $\eta\wedge\theta$ is equivalent to an atomic formula whenever $\eta,\theta$ are atomic.
\end{lemma}
\begin{proof}
By the assumptions and the normal form theorem, every first order formula is $\mathbb{T}$-equivalent
to a conjunction of formulas of the form $$\eta\vee\neg\theta_1\vee\cdots\vee\neg\theta_n$$
where $\eta,\theta_i$ are atomic. So, by the inclusion-exclusion principle, we must show that $\mu$ and $\nu$ coincide
on disjunctions of such formulas (which is again equivalent to one of the above form).
Again by the inclusion-exclusion principle, we must show that $\mu$ and $\nu$ coincide on formulas of the form
$$\eta\wedge\neg\theta_1\wedge\cdots\wedge\neg\theta_n.$$
This formula is itself $\mathbb{T}$-equivalent to a formula of the form $\eta\wedge\neg\theta$ where $\eta$ and $\theta$ are atomic.
For this last formula we have that
$$\mu(\eta\wedge\neg\theta)=\mu(\eta\vee\theta)-\mu(\theta)=\nu(\eta\vee\theta)-\nu(\theta)=\nu(\eta\wedge\neg\theta).$$
The second part is proved similarly.
\end{proof}

\begin{lemma}\label{existential values}
Let $\mathbb T$ be a complete and model-complete first order theory. Assume every disjunction of atomic formulas is
$\mathbb{T}$-equivalent to an atomic formula. Let $\mu$, $\nu$ be regular Borel probability measures on
$S_n(\mathbb{T})$ which coincide on affine infimal formulas.
Then, $\mu=\nu$.
\end{lemma}
\begin{proof}
By the assumptions, every first order formula is $\mathbb{T}$-equivalent to a conjunction of formulas of the form
$$\forall\y(\eta\vee\neg\theta_1\vee\cdots\vee\neg\theta_n)\ \ \ \ \ \ \ \ (*)$$
where $\eta(\x,\y)$, $\theta_i(\x,\y)$ are atomic.
So, by the inclusion-exclusion principle, we have to prove that $\mu$ and $\nu$ coincide on
disjunctions of such formulas. However, every such disjunction is again of the form $(*)$.
We may replace $\forall$ with $\inf$. So, we must prove that $\mu$ and $\nu$
coincide on any formula of the form
$$\inf_{\y}(\eta\vee\neg\theta_1\vee\cdots\vee\neg\theta_n)$$
where $\eta,\theta_i$ are atomic.
We may also identify $\neg\theta$ with $1-\theta$
and apply the Riesz space variant of the inclusion-exclusion principle.
So, $\mu$ and $\nu$ must coincide on any formula of the form
$$\inf_{\y}\sum_i r_i\phi_i$$ where every $\phi_i$ is of the general form
$$\eta\wedge\neg\theta_1\wedge\cdots\wedge\neg\theta_n.$$
Applying again the assumptions, every $\phi_i$ is of the form
$\eta\wedge\neg\theta\equiv_{\mathbb T}(\eta\vee\theta)-\theta$
where $\eta,\theta$ are atomic.
Again, $\eta\vee\theta$ is $\mathbb T$-equivalent to an atomic formula.
We conclude that any formula of the form $(*)$ is $\mathbb{T}$-equivalent
to a formula of the form $\inf_{\y}\phi$ where $\phi$ is affine and quantifier-free.
Since $\mu$ and $\nu$ coincide on such formulas, we conclude that $\mu=\nu$.
\end{proof}

Again, a similar result holds if $\eta\wedge\theta$ is $\mathbb{T}$-equivalent
to an atomic formula for every atomic $\eta,\theta$.

\begin{proposition} \label{main0}
Let $T$ be a complete affine theory such that $T^{\textrm{ex}}$ exists.
If $T$ has quantifier-elimination (resp. is model-complete) in the AL sense, then $T^{\textrm{ex}}$
has quantifier-elimination (resp. is model-complete) in the CL sense.
\end{proposition}
\begin{proof}
Let $p,q\in S_n(T^{\textrm{ex}})=E_n(T)$ be distinct. By Proposition \ref{separation},
there is an atomic formula $\theta$ such that $p(\theta)\neq q(\theta)$.
So, using the corresponding fact for CL, we conclude that $T^{\textrm{ex}}$ has quantifier-elimination.
The second part is similar. Assume $p(\phi)\neq q(\phi)$ where $\phi$ is an affine infimal formula.
Then, $\phi$ is infimal in the CL sense too. So, $p$ and $q$ are separated by an infimal formula in the CL sense
and hence $T^{\textrm{ex}}$ is model-complete.
An other way of proving this in the case $T^{\textrm{ex}}$ is first order is as follows.
For first order $M,N\vDash T$, assume $M\subseteq N$.
Then $M\preccurlyeq_{\mathrm{AL}}N$ and hence $M\preccurlyeq_{\mathrm{CL}}N$.
\end{proof}

\begin{theorem} \label{main}
Let $T$ be a complete affine theory such that $T^{\textrm{ex}}$ is first order.

(i) If $T^{\textrm{ex}}$ has quantifier-elimination and every disjunction (resp. conjunction) of atomic formulas is
$T^{\textrm{ex}}$-equivalent to an atomic formula, then $T$ has quantifier-elimination in the AL-sense.

(ii) If $T^{\textrm{ex}}$ is model-complete and every disjunction (resp. conjunction) of atomic formulas is
$T^{\textrm{ex}}$-equivalent to an atomic formula, then $T$ is model-complete in the AL-sense.
\end{theorem}
\begin{proof}
(i) By Lemma \ref{atomic values}, every regular Borel probability measure $\mu$ on
$S_n(T^{\textrm{ex}})$ is uniquely determined by its values on atomic formulas.
Let $p,q\in K_n(T)$ be distinct. By the Choquet-Bishop-de Leeuw theorem (see \cite{Alfsen}), $p,q$ are represented
by regular boundary measures, i.e. there are such measures $\mu$ and $\nu$ on $E_n(T)=S_n(T^{\textrm{ex}})$
such that for every affine formula $\phi$
$$p(\phi)=\int_{u\in E_n(T)}\hat\phi(u)d\mu$$ $$q(\phi)=\int_{u\in E_n(T)}\hat\phi(u)d\nu$$
Clearly, $\mu\neq\nu$ and hence there is an atomic formula $\theta$ such that $\mu(\theta)\neq\nu(\theta)$.
We conclude that $p(\theta)=\mu(\theta)\neq\nu(\theta)=q(\theta)$ and hence $T$ has quantifier-elimination.
The conjunction case is similar.

(ii) As in part (i), let $p,q\in K_n(T)$ be distinct types represented by boundary measures $\mu$ and $\nu$ respectively.
By Lemma \ref{existential values}, $\mu$ and $\nu$ differ on an affine infimal formula.
So, $p,q$ differ on that formula. We conclude that $T$ is model-complete. The conjunction case is similar.
\end{proof}
\vspace{1mm}

Let $\mathbb{T}$ be a complete first order theory and set $T=\mathbb{T}_{\textrm{af}}$.
Then $T^{\textrm{ex}}=\mathbb{T}$ and hence

\begin{corollary}
(i) Let $\mathbb{T}$ be a complete first order theory of fields in the language of rings $\{+,-,\times,0,1\}$.
If $\mathbb{T}$ has quantifier-elimination (resp. is model-complete) in the first order sense,
then $\mathbb{T}_{\textrm{af}}$ has quantifier-elimination (resp. is model-complete) in the AL sense.

(ii) Let $\mathbb{T}$ be a complete theory of Boolean algebras in the language $\{\wedge,\vee,\ ',0,1\}$
which has quantifier-elimination (resp. is model-complete).
Then, $\mathbb{T}_{\textrm{af}}$ has quantifier-elimination (resp. is model-complete).
\end{corollary}
\begin{proof}
(i): $(f=0)\vee(g=0)$ is equivalent to $fg=0$.

(ii): $x=y$ is equivalent to $(x\vee y)\wedge(x'\vee y')=0$.
Also, $(x=0)\wedge(y=0)$ is equivalent to $(x\vee y)=0$.
\end{proof}

\begin{example}
\em{(i) Algebraically closed fields and finite fields are the only division rings which have quantifier-elimination.
So, the affine part of ACF$_p$ as well as the affine theory of any finite field has quantifier-elimination.
The affine part of RCF is model-complete in the language of rings.

(ii) The theory of existentially closed difference fileds is model-complete but not complete.
So, the affine part of any completion of this theory is model-complete.

(iii) The affine part of the theory of atomless Boolean algebras has quantifier-elimination.
Also, the affine part of the theory of any finite Boolean algebra $B$ has quantifier-elimination.
In the case $B=\{0,1\}$, setting $\mu(x)=d(x,0)$, one has that
$$\mu(x\wedge y)+\mu(x\vee y)=\mu(x)+\mu(x)\ \ \ \ \ \ \ \ \ \ \forall xy.$$
So, in this special case, the affine part is the theory of probability algebras (see also \cite{Bagheri-Lip}).}

(iv) The theory of differential fields of characteristic $0$ (DCF$_0$)
has quantifier-elimination in the language $\{+,-,\times,\delta,0,1\}$.
So, by a similar reasoning, its affine part has quantifier elimination.
Likewise, DCF$_p$ and the theory of $(\Rn,+,\cdot,-,\exp,0,1)$ are model-complete.
So, their affine parts are model-complete.
\end{example}

By Lefschetz principle, ACF$_p$ tends to ACF$_0$ topologically. This property is
transferred to the affine parts. Let $T_p$ be the affine part of ACF$_p$ for $p=0$ or prime.
Since the language of fields is finite, there is a first order sentence $\eta$ such that for every
continuous structure $M$ one has that $M\vDash\eta=1$ if and only if $M$ is first order.
The following is then an affine variant of the Lefschetz principle.

\begin{proposition} \label{limit theory}
Let $\phi$ be an affine sentence in the language of fields. Then the following are equivalent:

(i) $T_0\vDash0\leqslant\phi$

(ii) For each $\epsilon>0$, there is $n$ such that for all $p\geqslant n$,\ \ $T_p\vDash-\epsilon\leqslant\phi$

(iii) For each $\epsilon>0$, there are arbitrarily large $p$ such that $T_p\vDash-\epsilon\leqslant\phi$
\end{proposition}
\begin{proof}
(i)$\Rightarrow$(ii): Given $\epsilon>0$, there is a condition $0\leqslant\sigma$ in $T_0$ such that
$0\leqslant\sigma\vDash-\frac{\epsilon}{2}\leqslant\phi$. So, there is a $\delta>0$ such that
$-\delta\leqslant\sigma\vDash-\epsilon\leqslant\phi$.
There is also a first order sentence $\eta\in ACF_0$ such that $\eta=1\vDash-\delta\leqslant\sigma$.
So, $\eta=1\vDash-\epsilon\leqslant\phi$. Let $n$ be such that $ACF_p\vDash\eta=1$ for every $p\geqslant n$.
Then, $ACF_p\vDash-\epsilon\leqslant\phi$ and hence $T_p\vDash-\epsilon\leqslant\phi$.

(ii)$\Rightarrow$(iii): Obvious.

(iii)$\Rightarrow$(i): Assume not. Then, there is $\epsilon>0$ such that $T_0\vDash\phi\leqslant-\epsilon$.
So, by the above argument, there is $n$ such that for all $p\geqslant n$, \ \
$T_p\vDash\phi\leqslant-\frac{1}{2}\epsilon$. This is a contradiction.
\end{proof}

An easy consequence of Proposition \ref{limit theory} is that if $\mu$ is an ultracharge
on the set $P$ of prime numbers such that $\mu(p)=0$ for every $p\in P$ and $M_p\vDash T_p$,
then $\prod_\mu M_p\vDash T_0$.

It was proves in \cite{Bagheri-Lip} that an affine theory in the pure language of metric spaces
having a classical model of cardinality at least $3$ can not have quantifier-elimination.
Since, the first order theory of infinite sets has quantifier-elimination,
we conclude that the conditions mentioned in Theorem \ref{main} (i) can not be dropped.

In all examples given above the language has no relation symbol other than equality.
There are many known examples of first order theories having additional relation symbols.
In such cases, the diversity of atomic formulas is highly increased making the situation hard.
As a final example, we prove that the affine part of the theory of
ordered divisible Abelian groups (ODAG) has quantifier-elimination.
For this purpose, we replace $\leqslant$ with the lattice operations
and show that any conjunction of atomic formulas is equivalent to an atomic one.
\bigskip

\noindent{\bf Affine ODAG}\\
Let $L=\{+,-,\wedge,\vee,0\}$ and $\mathbb T$ be the first order theory consisting of the following axioms:

(A1) Axioms of non-trivial torsion-free divisible Abelian groups (DAG)

(A2) Axioms of distributive lattices

(A3) $-(x\wedge y)=-x\vee -y$

(A4) $z+(x\wedge y)=(z+x)\wedge(z+y)$

(A5) $(x\wedge y=x)\vee(x\wedge y=y)$.
\bigskip

Let $M=(A,+,-,\wedge,\vee,0)$ be a first order model of $\mathbb T$. Set $x\leqslant y$ if $x\wedge y=x$.
Then $\bar{M}=(A,+,-,0,\leqslant)$ is a model of ODAG. Conversely, every model of ODAG is obtained in this way.
Let $M\subseteq N$ be models of $\mathbb T$. Then, $\bar{M}\subseteq\bar{N}$
and hence $\bar{M}\preccurlyeq\bar{N}$ by the model-completeness of ODAG.
Therefore, $M\preccurlyeq N$, since $\leqslant$ and lattice operations are bi-definable.
We conclude that $\mathbb T$ is model-complete (as well as complete).
Indeed, it has quantifier-elimination since it has algebraically prime models.
Note also that for $L$-terms $t_1,t_2$ one has that
$$(t_1=0)\wedge(t_2=0)\ \ \equiv\ \ (t_1\wedge-t_1\wedge t_2\wedge-t_2)=0.$$
By Theorem \ref{main}, the affine part of $\mathbb T$ has quantifier-elimination.

Let $M\vDash\mathbb{T}_{\textrm{af}}$ and define $x\leqslant y$ if $x\wedge y=x$.
This is a partial ordering on $M$.
A set $X\subseteq M^n$ is \emph{definable} if $d(\x,X)$ is the uniform limit of a sequence
$\phi^M_k(\x)$ of formulas (with parameters).
Let $$X=\{(x,y)\in M^2:\ x\leqslant y\}.$$
Then, members of $X$ are exactly pairs of the form $(x\wedge y,y)$ and
$$d((a,b),X)=\inf_{xy} d((a,b),(x\wedge y,y))$$
So, $X$ is definable. For every $a\leqslant b$, let $$[a,b]=\{t:\ a\leqslant t\leqslant b\}.$$
The intervals $[a,\infty)$, $(-\infty,b]$ and $(\infty,\infty)=M$ are defined similarly.
All these intervals are definable. For example,
$$d(x,[a,b])=\inf_t d(x,(t\vee a)\wedge b)=|(a-x)\vee(x-b)\vee0|$$
$$d(x,[a,\infty))=\inf_t d(x,t\vee a)=|(a-x)\vee0|.$$
Where $|x|=d(x,0)$. An axiomatization for $\mathbb{T}_{\textrm{af}}$ contains at least the following list of axioms
(denoted by $T$) where $x^+=x\vee0$ and $x^-=(-x)^+$.
\bigskip

(A1)-(A4) above

(A6) $|x-y|=d(x,y)$

(A7) $\sup_x|x\vee0|=1$,

(A8) $|x\wedge y|+|x\vee y|=|x|+|y|$

(A9) $|nx\wedge y|=|x\wedge y|$ \ \ \ \ \ for $n\geqslant1$

(A10) $|(x\vee y)^+|=|x^+\vee y^+|=|x^++y^+|$

(A11) $\inf_x|nx-y|=0$ \ \ \ \ \ for $n\geqslant1$ \ \ (divisibility)

(A12) $\inf_x\sum_{i=1}^n|x\wedge y_i|=\sum_{i=1}^n|0\wedge y_i|$

(A13) $\inf_t(|t\wedge x|-|t\wedge y|)=|0\wedge x|-|0\wedge y|+|y^+|-|y^+\wedge(x^+ + x^-)|$
\bigskip

Note that $\Qn$ satisfies all these axioms. A proof of quantifier-elimination for $T$
needs (A12)-(A13) and even more complicated axioms.

\begin{lemma}
Let $\a\in M\vDash T$. Then, every quantifier-free formula $\phi(x,\a)$ is
approximated by a formula of the form $\sum_{i=1}^nr_i\phi_i$ where $r_i\in\Qn$ and $\phi_i$
is one of forms $|x+a|$, $|(x+a)\wedge b|$, $|(x+a)\wedge(-x+b)|$ and $|(x+a)\wedge(-x+b)\wedge c|$ with $a,b,c\in M$.
\end{lemma}
\begin{proof}
We may assume $M$ is $\aleph_0$-saturated. Then, using the axioms, $\phi$ is approximated by a formula of the form
$$\sum_ir_i\cdot|\bigwedge_{j}(n_{ij}x+a_{ij})|$$ where $r_i\in\Qn$, $n_{ij}\in\Zn$ and $a_{ij}\in M$.
By (A11), we may write $nx+a=n(x+\frac{a}{n})$. By (A9), we may assume $n_{ij}=-1,0,1$.
Then, considering different cases, we obtain the possibilities mentioned in the lemma.
\end{proof}

\begin{proposition}
Let $M\vDash T$ be $\aleph_0$-saturated. Then, $\Qn\subseteq M$. If $M\vDash\mathbb{T}_{\textrm{af}}$, then $\Qn\preccurlyeq M$.
\end{proposition}
\begin{proof}
By (A7), there is $a\in M$ such that $0<a$ and $|a|=1$.
We have therefore that $ra\leqslant sa$ whenever $r\leqslant s$ are rationals.
We conclude that $r\mapsto ra$ is an (isometric) $L$-embedding from $\Qn$ into $M$.
If $M\vDash\mathbb{T}_{\textrm{af}}$, this is an elementary embedding by model-completeness.
\end{proof}
\bigskip

One proves in a similar way that the affine part of the first order theory of
the structure $(\Rn,+,\times,-,\wedge,\vee,0,1)$ has quantifier-elimination.
The affine part of ODAG (or its subtheories) is some sense the bounded version of the theory of Banach lattices.
Likewise, the affine parts of DAG and RCF/ACF have similarities with Banach spaces
and real/complex Banach algebras respectively.

The classical theories stated above are all decidable. So, their affine parts are decidable as well.
This is because affine formulas with integer coefficients form a computable part of the
classical formulas (if they are presented in the Riesz space form).
On the other hand, giving a precise axiomatization of these affine parts is not so easy.
So, we leave open the following question.
\bigskip

\noindent{\bf Question:} Find a complete axiom system for the affine part of classical theories
such as ODAG, ACF$_p$ and RCF.

\end{document}